\documentclass[reqno]{amsart}

\usepackage{hyperref}
\usepackage{amsmath,amsthm,amssymb}

\newtheorem{thm}{Theorem}
\newtheorem{cor}[thm]{Corollary}
\newtheorem{prop}[thm]{Proposition}
\newtheorem{lem}[thm]{Lemma}
\newtheorem{clm}[thm]{Claim}

\theoremstyle{definition}
\newtheorem{dfn}[thm]{Definition}
\newtheorem{rem}[thm]{Remark}
\newtheorem{ex}[thm]{Example}
\newtheorem{conj}[thm]{Conjecture}
\newtheorem{conv}[thm]{Convention}

\numberwithin{thm}{section}
\numberwithin{equation}{section}

\DeclareMathOperator{\vol}{vol}
\DeclareMathOperator{\Ric}{Ric}
\DeclareMathOperator{\Sc}{Sc}
\DeclareMathOperator{\Hess}{Hess}
\DeclareMathOperator{\di}{div}
\DeclareMathOperator{\II}{II}

\newcommand{\step}[1]{\medskip\noindent\textit{#1.}}
\newcommand{\pack}[1]{{#1}\mathchar`-\mathrm{pack}}

\title[A lower bound for the curvature integral]{A lower bound for the curvature integral\\under an upper curvature bound}
\author[T. Fujioka]{Tadashi Fujioka}
\address{Department of Mathematics, Osaka University, Toyonaka, Osaka 560-0043, Japan}
\email{fujioka@cr.math.sci.osaka-u.ac.jp, tfujioka210@gmail.com}
\date{\today}
\thanks{Supported by JSPS KAKENHI Grant Number 22KJ2099}

\begin{document}

\begin{abstract}
We prove that the integral of scalar curvature over a Riemannian manifold is uniformly bounded below in terms of its dimension, upper bounds on sectional curvature and volume, and a lower bound on injectivity radius.
This is an analogue of an earlier result of Petrunin for Riemannian manifolds with sectional curvature bounded below.
\end{abstract}

\maketitle

\section{Introduction}\label{sec:intro}

\subsection{Main result}\label{sec:main}

In this paper we prove the following.
All manifolds are assumed to be complete and of dimension $\ge 2$.
Let $\Sc$ denote scalar curvature.

\begin{thm}\label{thm:main}
Let $M$ be an $n$-dimensional Riemannian manifold with sectional curvature $\le 1$, injectivity radius $\ge r>0$, and volume $\le V$.
Then we have
\[\int_{M}\Sc\ge C(n,r,V),\]
where $C(n,r,V)$ is a (negative) constant depending only on $n$, $r$, and $V$.
\end{thm}

This is an analogue of the following result due to Petrunin \cite{Pet:int} (cf.\ \cite[\S6]{Li:c2}).

\begin{thm}[Petrunin \cite{Pet:int}]\label{thm:pet}
Let $M$ be an $n$-dimensional Riemannian manifold with sectional curvature $\ge-1$ and diameter $\le D$.
Then we have
\[\int_{M}\Sc\le C(n,D),\]
where $C(n,D)$ is a (positive) constant depending only on $n$ and $D$.
\end{thm}

Our proof parallels (and owes much to) Petrunin's one: both rely on the convergence theory of Riemannian manifolds with one-sided curvature bounds, especially its rescaling techniques.
The parallel structure of the proofs comes from that of the limit spaces, i.e., Alexandrov spaces and GCBA spaces (see Section \ref{sec:ad} below).

In Theorem \ref{thm:main}, neither the injectivity radius bound nor the volume bound can be dropped.
In either case, one can construct a surface of genus $g$ satisfying any other bound, but the Gauss-Bonnet theorem tells us that its total scalar curvature is $4\pi(2-2g)$, which goes to $-\infty$ as $g\to\infty$.
For example, a connected sum of flat tori getting smaller and smaller admits upper bounds on curvature and volume, but no lower bound on injectivity radius.

The lower bound on total scalar curvature, together with the upper bounds on sectional curvature and volume, immediately yields

\begin{cor}
Under the same assumptions as in Theorem \ref{thm:main}, there is a uniform $L^1$ bound on the curvature tensor of $M$ in terms of $n$, $r$, and $V$.
\end{cor}

Recently, Lebedeva-Petrunin \cite{LP:curv} obtained a convergence theorem for the curvature tensors of Riemannian manifolds with sectional curvature bounded below and without collapsing.
They regarded curvature tensors as measures and showed the weak convergence under the Gromov-Hausdorff convergence.
It is natural to expect that the same holds true for the case of upper curvature bound (where no collapsing occurs).

\begin{thm}[Lebedeva-Petrunin \cite{LP:curv}]
Suppose a sequence $M_j$ of $n$-dimensional Riemannian manifolds with sectional curvature $\ge-1$ converges to an Alexandrov space $X$ of dimension $n$.
Then the curvature tensor of $M_j$ weakly converges to a ``measure-valued tensor'' on $X$.
\end{thm}

\begin{conj}
Suppose a sequence $M_j$ of $n$-dimensional Riemannian manifolds with sectional curvature $\le1$ and injectivity radius $\ge r>0$ converges to a GCBA space $X$.
Then the curvature tensor of $M_j$ weakly converges to a ``measure-valued tensor'' on $X$.
\end{conj}

In particular, the total scalar curvature of $M_j$ will converge, if $X$ is compact.
Furthermore, if $X$ is Riemannian, the limit value will be the total scalar curvature of $X$.
In the case of lower curvature bound, the latter is an unpublished result of Perelman (see \cite[App.\,B]{Pet:poly}).
The $2$-dimensional case follows from a classical result of Alexandrov-Zalgaller \cite[Ch.\,VII \S4.13]{AZ}.

More recent results in the nonsmooth setting of Alexandrov spaces can be found in \cite{Li:c2}.
See \cite{A:conj} for the conjectural behavior of total scalar curvature in the case of collapse.
See also \cite{Nab} for related conjectures in the Ricci lower bound case.

\subsection{Analogies and differences}\label{sec:ad}

The purpose of this section is to elaborate on the analogies and differences between the proofs of Theorems \ref{thm:main} and \ref{thm:pet}.

In either case, the most fundamental fact is that the family of manifolds under consideration is precompact in the Gromov-Hausdorff topology.
The limit spaces are metric spaces with synthetic sectional curvature bounds.
In the case of lower curvature bound, it is an Alexandrov space (= CBB space) in the sense of Burago-Gromov-Perelman \cite{BGP}.
In the case of upper curvature bound, it is a GCBA space (= geodesically complete CBA space) in the sense of Lytchak-Nagano \cite{LN:geo}.
These spaces have different properties in general, but to some extent they have parallel structures.
For example, shortest paths in an Alexandrov space are neither unique nor extendable, but do not branch, whereas shortest paths in a GCBA space are unique and extendable (within injectivity radius), but do branch.
This duality enables us to prove Theorem \ref{thm:main} in a parallel way to Theorem \ref{thm:pet}.
Although there is no collapse in the CBA case, we still need a rescaling technique similar to the one used in the CBB case, and also different technical difficulties arise.
There are three main differences in the proof that the author is aware of.

First of all, the logical structure of the proof is exactly the same for CBB and CBA.
The proof is by induction on dimension, actually not for a manifold itself, but for a special submanifold, called strained surface (technically, we need a slightly modified one, called corner surface as in \cite[3.1]{Pet:int}).

The base case is $2$-dimensional, where by the Gauss-Bonnet theorem the total scalar curvature is essentially equal to the Euler characteristic.
In the CBB case, there is nothing to prove since the Euler characteristic of a surface is at most $2\pi$.
In the CBA case, we have to bound the Euler characteristic from below, which is possible because of the absence of collapse, by a stability result of Lytchak-Nagano \cite[13.1]{LN:geo}.
This is the first difference.
(However, such a lower bound also exists in the CBB case, despite the lack of stability due to collapse.)

The proof of the higher-dimensional case has two important steps.
Both involve technical differences between CBB and CBA geometries.
The convergence theory is used in Step (1) and Step (2) only concerns a fixed manifold.

\begin{enumerate}
\item Bounding the bad behavior of geodesics (by precompactness).
\item Calculating the curvature integral (by induction).
\end{enumerate}

As for Step (1), first recall that the CBB/CBA condition is nothing but the contraction/divergence property of geodesics.
In a CBB Riemannian manifold, if geodesics emanating from a point are more contracting than those in a metric cone, then the volume of a metric sphere centered at that point, normalized by its radius, will strictly decrease.
This argument extends directly to general Alexandrov spaces, as in \cite[3.7]{Pet:int}.
The opposite statement seems to be true for a CBA Riemannian manifold, that is, the more diverging geodesics emanating from a point, the bigger the volume of a normalized sphere.
This is indeed true, but the proof requires a small trick using the geodesic flow, which cannot be extended to general GCBA spaces (see Remark \ref{rem:res} and Example \ref{ex:res}).
This is the second difference.

Based on the above observation, one can prove a CBA analogue of the rescaling theorem of Petrunin \cite[3.6]{Pet:int} (Theorem \ref{thm:res}).
This rescaling theorem, together with the precompactness, allows us to cover our manifold by a uniformly finite number of ``good'' metric annuli on which geodesics from each center behave like those in a metric cone.
This is where the uniform boundedness comes from.
A similar good annulus covering was used in the more recent paper by Li \cite[3.2]{Li:c2}.
(Nevertheless, in the actual proof, as in the CBB case, we will not take such a good annulus covering, but apply the rescaling theorem directly to an argument by contradiction.)

In Step (2), we calculate the curvature integral of such a good annulus inductively from those of the concentric metric spheres (Proposition \ref{prop:ind}).
We make use of the Bochner-type formula \eqref{eq:boc} developed by Petrunin to calculate the curvature integral in the normal direction to the sphere.
This is the same for CBB and CBA.
The difference is rather in the tangent direction, especially the extrinsic curvature, which is caused by the difference between the semiconcavity and (semi)convexity of distance functions in CBB and CBA manifolds.
Compare Claim \ref{clm:conv} with \cite[4.5(ii)c]{Pet:int}.
Moreover, in Claim \ref{clm:G}, we prove that for a small metric sphere $L$ in our manifold $M$,
\[K_L^+\le K_M^++G,\]
where $K_L^+$ and $K_M^+$ denote the maximal sectional curvatures of $L$ and $M$, respectively, and $G$ denotes the Gauss-type extrinsic curvature of $L$ in $M$ (see Section \ref{sec:boc}).
Compare this inequality with the Trivial inequality (ii) in \cite[4.5]{Pet:int} for the opposite estimate in the CBB case, where the mean curvature $H$ appears instead of $G$.
It seems impossible to bound $G$ itself from above, but the Bochner formula mentioned above allows us to bound the integral of $G$ in terms of the curvature integral of the ambient space $M$.
The latter is bounded above by the assumption.
This is the third difference.

\begin{rem}\label{rem:ad}
In CBB geometry, the fact that a convex hypersurface inherits the lower curvature bound of the ambient manifold is frequently used for induction, especially in the noncollapsing setting (see for instance, \cite[App.\,B]{Pet:poly}, \cite{LP:curv}, and \cite{K:reg}).
This is not the case for a CBA manifold, but the last paragraph suggests that it does not matter here.
This might remind the reader of the difference between the proofs of \cite[1.4]{K:reg} and \cite[1.3]{LN:top}.
\end{rem}

\step{Organization}
In Section \ref{sec:pre}, we recall a variant of the Bochner formula developed by Petrunin \cite{Pet:int} and the notions of a GCBA space and a strainer introduced by Lytchak-Nagano \cite{LN:geo}.
Since our definition of a strainer is slightly different from the original one, we need to prove some basic lemmas on it.
In Section \ref{sec:res}, we prove a key rescaling theorem \ref{thm:res}, which is a CBA analogue of \cite[3.6]{Pet:int}.
Section \ref{sec:prf} is devoted to the proof of Theorem \ref{thm:main}, by induction on the dimension of a strained surface.
To exhaust a strained surface by lower-dimensional ones, we carry out another induction based on the rescaling theorem.
Finally in Section \ref{sec:app}, we show some auxiliary results used in the previous section.

\step{Acknowledgments}
I would like to thank Professors Anton Petrunin and Alexander Lytchak for their interest in my work and helpful comments.
Both gave me the same comments regarding Remark \ref{rem:ad}, which motivated me to write Section \ref{sec:ad}.
Alexander Lytchak also provided a simplified proof of Proposition \ref{prop:vol}.
I am also grateful to my former colleague Ya Gao for discussion and encouragement.

\section{Preliminaries}\label{sec:pre}

\subsection{Notation}\label{sec:not}

The distance between $x$ and $y$ is denoted by $|xy|$.
The open and closed $r$-balls around $p$ and its boundary are denoted by $B(p,r)$, $\bar B(p,r)$, and $\partial B(p,r)$, respectively.
The open metric annulus $B(p,r)\setminus\bar B(p,s)$ is denoted by $A(p;s,r)$.
We allow $s=0$, in which case $A(p;0,r)=B(p,r)\setminus\{p\}$.
For a metric space $X$ and $\lambda>0$, we denote by $\lambda X$ the rescaled space with metric multiplied by $\lambda$.
The sectional, Ricci, and scalar curvatures are denoted by $K$, $\Ric$, and $\Sc$, respectively, with subscript indicating a manifold if necessary.
All integrals are with respect to the Riemannian volume measure, denoted by $\vol$, which is usually omitted.
We will use the notation $C$ to denote various (usually positive) constants, which depend only on indicated parameters.

\subsection{Bochner formula}\label{sec:boc}

We recall the Bochner-type formula developed in \cite{Pet:int}.

We will use the same notation as in \cite[\S2]{Pet:int}.
Let $M$ be an $(m+1)$-dimensional Riemannian manifold.
Assume that a smooth function $f:M\to\mathbb R$ has no critical values on $[a,b]\subset\mathbb R$ and that its level sets $L_t=f^{-1}(t)$ are compact for all $t\in[a,b]$.
We denote
\begin{itemize}
\item by $u=\nabla f/|\nabla f|$ the unit normal field to $L=L_t$;
\item by $\kappa_i(x)$ ($1\le i\le m$) the principal curvatures of $L$ at $x\in L$, that is, the eigenvalues of the shape operator $S_x:T_xL\to T_xL$ defined by $S_x(v)=\nabla_vu$ for $v\in T_xL$;
\item by $H(x)=\sum_{i=1}^m\kappa_i(x)$ the mean curvature of $L$ at $x$;
\item by $G(x)=\sum_{i\neq j}\kappa_i(x)\kappa_j(x)$ the Gauss-type curvature of $L$ at $x$, defined as the external term in the Gauss formula for the scalar curvature of $L$ at $x$, i.e.,
\begin{equation}\label{eq:gau}
\Sc_{L}=\Sc_M-2\Ric_M(u,u)+G.
\end{equation}
\end{itemize}

The following is a special form of the integral Bochner formula, which gives us the integral of the normal Ricci curvature in the above Gauss formula in terms of the integral of the extrinsic curvatures.

\begin{lem}[{\cite[2.2]{Pet:int}}]
In the situation above, we have
\begin{equation}\label{eq:boc}
\int_{f^{-1}[a,b]}\Ric_M(u,u)=\int_{f^{-1}[a,b]}G+\int_{L_a}H-\int_{L_b}H.
\end{equation}
\end{lem}

\subsection{GCBA spaces}\label{sec:gcba}

We refer to \cite{LN:geo} and \cite{LN:top} for the general theory of GCBA spaces.
Here we recall the basic definitions and properties, focusing on the Riemannian case and  their limit spaces.

A \textit{CBA($\kappa$) space} is a metric space with curvature bounded above by $\kappa$ locally in the sense of triangle comparison.
More precisely, any small geodesic triangle is not thicker than the comparison triangle with the same sidelengths in the model plane of constant curvature $\kappa$.
A \textit{CAT($\kappa$) space} is a metric space where the triangle comparison holds globally, as long as the perimeter of a triangle is less than twice the diameter of the model $\kappa$-plane.
For a Riemannian manifold, CBA($\kappa$) is equivalent to the sectional curvature being bounded above by $\kappa$, and CAT($\kappa$) is equivalent to, in addition to CBA($\kappa$), the injectivity radius being bounded below by that of the model $\kappa$-plane.

A CBA space is \textit{locally geodesically complete} if any geodesic (= local shortest path) is locally extendable.
If the CBA space is complete, then local geodesic completeness is equivalent to global geodesic completeness, that is, any geodesic is infinitely extendable (a kind of the Hopf-Rinow theorem).
A \textit{GCBA space} is a separable, locally compact, locally geodesically complete CBA space.

For a point $p$ in a GCBA space, the \textit{space of directions} and \textit{tangent cone} at $p$ will be denoted by $\Sigma_p$ and $T_p$, respectively.
That is, $\Sigma_p$ is the set of shortest paths emanating from $p$ equipped with the angle metric, and $T_p$ is the Euclidean cone over $\Sigma_p$, which is isometric to the blow-up of the space at $p$.
Note that $\Sigma_p$ is a compact geodesically complete CAT($1$) space, and $T_p$ is a locally compact geodesically complete CAT($0$) space.
For a Riemannian manifold, they are nothing but the unit tangent sphere and the tangent space.

The \textit{dimension} of a GCBA space is defined by its Hausdorff dimension, which coincides with the topological dimension.
The local dimension at each point (= the dimension of a small neighborhood or the tangent cone) is finite, but not constant in general.
If it is constant, the GCBA space is called \textit{pure-dimensional}.

We will now focus on the Riemannian case.
Let $\mathcal M(n,r,V)$ denote the family of $n$-dimensional (complete) Riemannian manifolds with sectional curvature $\le 1$, injectivity radius $\ge r>0$, and volume $\le V$, as in Theorem \ref{thm:main}.
Throughout the paper, we assume that $r$ is small enough compared to $\pi$, the injectivity radius of the unit sphere $\mathbb S^n$.

A metric space is totally bounded if for any $\varepsilon>0$ it contains a finite $\varepsilon$-net, that is, a subset whose $\varepsilon$-neighborhood is the whole space.
A family $\mathcal M$ of metric spaces is \textit{uniformly totally bounded} if for any $\varepsilon>0$ every element of $\mathcal M$ contains a uniformly finite $\varepsilon$-net.
For any $M\in\mathcal M(n,r,V)$, the G\"unther volume comparison (cf.\ \cite[6.1]{N:vol}) says that the volume of any $s$-ball in $M$ ($s\le r$) is not less than the volume of the $s$-ball in $\mathbb S^n$.
Together with the upper volume bound $V$, this implies that $\mathcal M(n,r,V)$ is uniformly totally bounded.
In particular, every element of $\mathcal M(n,r,V)$ is compact and has uniformly bounded diameter.

A metric space is $C$-doubling if any metric ball can be covered by at most $C$ balls of half radius.
A family $\mathcal M$ of metric spaces is \textit{uniformly $C$-doubling} if  every element of $\mathcal M$ is $C$-doubling for fixed $C$.
For any $M\in\mathcal M(n,r,V)$, the relative volume comparison \`a la Bishop-Gromov (cf.\ \cite[6.3]{N:vol}) shows that the volume of any $s$-ball in $M$ ($s\le r$) is bounded above by a constant multiple of  the volume of the $s$-ball in $\mathbb S^n$, where the constant depending only on $n$, $r$, and $V$.
Together with the absolute volume comparison of G\"unther mentioned above, this implies that $\mathcal M(n,r,V)$ is uniformly $C$-doubling, where $C$ is a constant depending only on $n$, $r$, and $V$ (cf.\ \cite[\S5.2]{LN:geo}).

The uniform total boundedness implies that $\mathcal M(n,r,V)$ is \textit{precompact} in the Gromov-Hausdorff topology, meaning that any sequence in $\mathcal M(n,r,V)$ has a convergent subsequence.
The limit is a compact GCBA space of pure dimension $n$ and with the same geometric bounds (for curvature, injectivity radius, and volume \cite{N:vol}).
Moreover, by the stability theorem of Lytchak-Nagano \cite[1.3]{LN:top}, it turns out that the limit is a topological manifold homeomorphic to the elements of the sequence with sufficiently large indices.

In this paper we basically deal with such GCBA spaces arising as Gromov-Hausdorff limits of elements of $\mathcal M(n,r,V)$.
We denote by $\bar{\mathcal M}(n,r,V)$ the Gromov-Hausdorff closure of $\mathcal M(n,r,V)$.

\subsection{Strainers}\label{sec:str}

We recall the notion of a strainer introduced in \cite{LN:geo} (originally in the CBB setting \cite{BGP}).
More precisely, we introduce a slightly stronger version and prove some lemmas on it.
Although we restrict our attention to GCBA spaces in $\bar{\mathcal M}(n,r,V)$, the contents of this section are applicable to general GCBA spaces.

Let $N$ be a GCBA space in $\bar{\mathcal M}(n,r,V)$.
As mentioned earlier, we assume $r\ll\pi$.
In what follows we basically work in a closed ball of radius $r/10$ in $N$, which plays the same role as a \textit{tiny ball} in \cite{LN:geo}.
By triangle comparison, such a ball is convex, and shortest paths are unique and extendable in it.

A strainer introduced in \cite{LN:geo} controls the branching of shortest paths in a GCBA space, so it is an infinitesimal notion in nature.
However, our main concern is a Riemannian manifold, where geodesics never branch (but may diverge rapidly).
For this reason, we need to define a strainer with the notion of \textit{straining radius} as in \cite[\S7.5]{LN:geo}, which controls the local behavior of geodesics (however, the original infinitesimal definition has other advantages; for instance, see Remark \ref{rem:vol3}).
Also, we will define it using comparison angles and non-strict inequalities rather than actual angles and strict inequalities, so that our strainers are stable under the Gromov-Hausdorff convergence (see Remark \ref{rem:str}(1) below).

Let $k$ be a positive integer and $\delta$ and $\varepsilon$ positive numbers.
We always assume $\delta\ll1/n$ and $\varepsilon<r/10$ (see also Convention \ref{conv:d}).
Denote by $\tilde\angle$ comparison angle.

\begin{dfn}\label{dfn:str}
Let $p$ be a point in a GCBA space $N\in\bar{\mathcal M}(n,r,V)$.
A collection $\{a_i\}_{i=1}^k$ of points in $\bar B(p,r/10)\setminus B(p,\varepsilon)$ is called a \textit{$(k,\delta,\varepsilon)$-strainer} at $p$ if there exists another collection $\{b_i\}_{i=1}^k$ of points in $\bar B(p,r/10)\setminus B(p,\varepsilon)$ (called an \textit{opposite strainer}) such that the following hold for any distinct $x,y\in B(p,\varepsilon)$:
\begin{gather}
\tilde\angle a_ixy+\tilde\angle b_ixy\le\pi+\delta,\label{eq:str1}\\
\tilde\angle a_ixa_j,\tilde\angle a_ixb_j,\tilde\angle b_ixb_j\le\pi/2+\delta,\label{eq:str2}
\end{gather}
where $1\le i\neq j\le k$.
The value $\varepsilon$ is called the \textit{straining radius} of this strainer.
The map $F=(|a_1\cdot|,\dots,|a_k\cdot|)$ is called a \textit{$(k,\delta,\varepsilon)$-strainer map} at $p$.

For $U\subset N$, we say that $F:N\to\mathbb R^k$ is a \textit{$(k,\delta,\varepsilon)$-strainer map} on $U$ if it is a $(k,\delta,\varepsilon)$-strainer map at any $p\in U$ (note that the opposite strainer may be different for each $p$).
For convenience, we regard a constant map $F:N\to\mathbb R^0$ as a $(0,\delta,\varepsilon)$-strainer map for any $\delta$ and $\varepsilon$.
\end{dfn}

The inequality \eqref{eq:str1} means that any extension of the shortest path from $a_i$ to $x$ goes in almost the same direction as the direction to $b_i$.
In other words, shortest paths in that direction diverge very little.
The inequality \eqref{eq:str2} means that these almost unique opposite pairs $(a_i,b_i)$ are almost orthogonal.
The point here is that the opposite strainer is the same for every point in the ball of straining radius.

The above notions of a strainer and straining radius are obviously stronger than the ones in \cite[7.2, 7.10]{LN:geo}, in that ours are defined by comparison angles.
Moreover, our strainer has the following two properties (especially the first one).

\begin{rem}\label{rem:str}
Suppose $N_j\to N$ is a Gromov-Hausdorff convergent sequence in $\bar{\mathcal M}(n,r,V)$.
Suppose $p_j,a_i^j\in N_j$ converge to $p,a_i\in N$ ($1\le i\le k$), respectively.
\begin{enumerate}
\item(Stability) If $\{a_i^j\}_{i=1}^k$ is a $(k,\delta,\varepsilon)$-strainers at $p_j$, then so is $\{a_i\}_{i=1}^k$ at $p$.
\item(Liftability) If $\{a_i\}_{i=1}^k$ is a $(k,\delta,\varepsilon)$-strainer at $p$, then $\{a_i^j\}_{i=1}^k$ is a $(k,2\delta,\varepsilon/2)$-strainer at $p_j$ for sufficiently large $j$.
\end{enumerate}
\end{rem}

The first one is trivial from the definition.
The second one follows from geodesic completeness and angle monotonicity.
Indeed, these imply that one may assume $|xy|>\varepsilon/2$ in the inequality \eqref{eq:str1} for $x\in B(p,\varepsilon/2)$.
Thus the inequality \eqref{eq:str1} can be lifted to the converging sequence as well as the inequality \eqref{eq:str2}.

The rest of this section contains a few basic lemmas on our strainer.
In general, a strainer arises from the fact that the tangent cone of a GCBA space is a metric cone.
The following lemma is a quantitative version of \cite[7.3]{LN:geo}.

\begin{lem}\label{lem:str1}
Let $p\in N$ in $\bar{\mathcal M}(n,r,V)$.
Then for any $\delta>0$, there exists $\rho>0$ such that $p$ is a $(1,\delta,10^{-1}\delta|px|)$-strainer at any $x\in B(p,\rho)\setminus\{p\}$.
\end{lem}

\begin{proof}
The proof is essentially the same as \cite[7.3]{LN:geo}.
Since the tangent cone $T_p$ at $p$ is a metric cone, it is easy to see that the vertex $o$ is a $(1,2^{-1}\delta,5^{-1}\delta|ov|)$-strainer at $v\in T_p\setminus\{o\}$ for any $\delta>0$ (note that we regard the injectivity radius of $T_p$ as $\infty$ since it is CAT($0$), which is different from our setting above, but easy to modify).
Indeed, by rescaling, we may assume $|ov|=1$.
Then a point $\bar v$ sufficiently far from $o$ on the ray emanating from $o$ in the direction $v$ will be an opposite strainer. 
This is because for any $w\in B(v,5^{-1}\delta)$ the shortest path from $w$ to $\bar v$ makes an angle almost $5^{-1}\delta$ with the ray emanating from $o$ in the direction $w$.
Now suppose the statement does not hold.
Then there exists a sequence $x_i\in N$ converging to $p$ at which $p$ is not a $(1,\delta,10^{-1}\delta|px_i|)$-strainer.
Since $(|px_i|^{-1}N,p)$ converges to $(T_p,o)$, we get a contradiction to the liftability of a strainer, Remark \ref{rem:str}(2).
\end{proof}

The existence of a strainer controls the divergence rate of geodesics (even in a Riemannian manifold where no geodesics branch).
The next lemma generalizes the observation used in the above proof for the tangent cone.
Let $0<\rho<r/10$.

\begin{lem}\label{lem:str2}
Let $p\in N$ in $\bar{\mathcal M}(n,r,V)$.
For any $x\in B(p,\rho)\setminus\{p\}$, let $\bar x$ denote a point at distance $\rho$ from $p$ on an extension of the shortest path $px$ beyond $x$.
If $p$ is not a $(1,\delta,\varepsilon)$-strainer at some $x\in B(p,\rho/10)\setminus\{p\}$, where $\varepsilon<|px|$, then there exists $y\in B(x,\varepsilon)$ such that $|\bar x\bar y|>\rho\delta/2$.
\end{lem}

\begin{proof}
Suppose $|\bar x\bar y|\le\rho\delta/2$ for any $y\in B(x,\varepsilon)$.
We show that $p$ is a $(1,\delta,\varepsilon)$-strainer at $x$ with an opposite strainer $\bar x$.
Note that $\bar x$ and $\bar y$ are both at distance $>4\rho/5$ from $y$.
For any other $z\in B(x,\varepsilon)$, this implies that $\tilde\angle\bar xyz$ is $\delta$-close to $\tilde\angle\bar yyz$ (note $\rho\ll\pi$).
Since $py\bar y$ is a shortest path, angle comparison shows $\tilde\angle pyz+\tilde\angle\bar yyz\ge\pi$.
Therefore
\[\tilde\angle pyz+\tilde\angle\bar xyz\ge\pi-\delta.\]
On the other hand, since we are in a tiny ball of radius $r\ll\pi$,
\[\tilde\angle pyz+\tilde\angle pzy\le\pi,\quad\tilde\angle\bar xyz+\tilde\angle\bar xzy\le\pi.\]
Combining the three inequalities yields
\[\tilde\angle pzy+\tilde\angle\bar xzy\le\pi+\delta,\]
which is the desired condition \eqref{eq:str1} of a strainer.
\end{proof}

A strainer map can be extended by adding a distance map from another strainer in its fiber (cf.\ \cite[9.4]{LN:geo}).

\begin{lem}\label{lem:str3}
Let $F$ be a $(k,\delta,\varepsilon)$-strainer map at $p\in N$ in $\bar{\mathcal M}(n,r,V)$.
If $q\in B(p,\varepsilon)\cap F^{-1}(F(p))$ is a $(1,\delta,\varepsilon'|pq|)$-strainer at $p$, where $\varepsilon'\le\delta$, then $(F,|q\cdot|)$ is a $(k+1,10\delta,\varepsilon'|pq|)$-strainer map at $p$.
\end{lem}

\begin{proof}
We may assume $k=1$.
Let $a$ be the $(1,\delta,\varepsilon)$-strainer at $p$ defining $F$ with an opposite strainer $b$.
The assumption $F(p)=F(q)$ implies $\tilde\angle apq=\tilde\angle aqp$.
Since the sum of of these two comparison angles in a tiny ball is at most $\pi$, we have $\tilde\angle apq\le\pi/2$.
Hence it suffices to show $\tilde\angle bpq\le\pi/2+\delta$.
Indeed, since $\varepsilon'\le\delta$, these imply $\tilde\angle axq,\tilde\angle bxq\le\pi/2+10\delta$ for any $x\in B(p,\varepsilon'|pq|)$.

Extend the shortest path $aq$ a bit beyond $q$ and let $\bar a$ be the endpoint.
Then the strainer condition \eqref{eq:str1} implies $\tilde\angle bq\bar a\le\delta$ and hence $\angle bq\bar a\le\delta$.
Therefore
\[\tilde\angle aqp+\tilde\angle bqp\ge\angle aqp+\angle bqp\ge\angle aqb\ge\angle aq\bar a-\angle bq\bar a\ge\pi-\delta.\]
Since $\tilde\angle aqp\le\pi/2$, we have $\tilde\angle bqp\ge\pi/2-\delta$.
On the other hand, since we are in a tiny ball, $\tilde\angle bpq+\tilde\angle bqp\le\pi$.
Therefore $\tilde\angle bpq\le\pi/2+\delta$, as desired.
\end{proof}

Finally, let us consider the Riemannian case, i.e., $N\in\mathcal M(n,r,V)$.
Let $F=(f_1,\dots,f_k)$ be a $(k,\delta,\varepsilon)$-strainer map at $p\in N$, associated with a strainer pair $\{(a_i,b_i)\}_{i=1}^k$.
The injectivity radius bound implies that $f_i$ is smooth and regular (its gradient has norm $1$) on $B(p,\varepsilon)$, and also the upper curvature bound implies that $f_i$ is strictly convex on $B(p,\varepsilon)$.
Furthermore, their gradients $\nabla f_i$ (opposite to the directions to $a_i$) are almost orthogonal.
Indeed, the inequality \eqref{eq:str1}, together with angle comparison, implies that $\nabla f_i$ makes an angle no greater than $\delta$ with the direction to $b_i$.
Then the inequality \eqref{eq:str2}, together with angle comparison, shows that the angle between $\nabla f_i$ and $\nabla f_j$ is almost $\pi/2$ up to error $10\delta$ for every $i\neq j$.
In particular, the map $F$ is regular, and hence any fiber of $F$ is a smooth (open) submanifold in $B(p,\varepsilon)$.

\section{Rescaling}\label{sec:res}

We prove a key rescaling theorem, which is a CBA analogue of \cite[3.6]{Pet:int}.

Let $\Sigma$ be a compact metric space and $\delta>0$.
The \textit{$\delta$-packing number} of $\Sigma$ is the maximal number of points in $\Sigma$ with pairwise distance at least $\delta$.
We denote it by $\pack\delta(\Sigma)$.

Suppose $\Sigma$ is a space of directions of a GCBA space in $\bar{\mathcal M}(n,r,V)$.
Then
\[\pack\delta(\mathbb S^{n-1})\le\pack\delta(\Sigma)\le C(n,r,V,\delta),\]
where $C(n,r,V,\delta)$ is a constant depending only on $n$, $r$, $V$, and $\delta$.
The left inequality follows from the fact that $\Sigma$ admits a $1$-Lipschitz surjective map to $\mathbb S^{n-1}$ (\cite[11.3]{LN:geo}).
The right inequality follows, for example, from the upper semicontinuity of the spaces of directions (\cite[5.13]{LN:geo}) and the compactness of $\bar{\mathcal M}(n,r,V)$.

The rescaling theorem below enables us to cover our manifold by a uniformly finite number of strained annuli, by induction on the packing number of the space of directions of the limit GCBA space.

\begin{conv}\label{conv:d}
From now on, we assume that $\delta$ is less than some small constant $\delta_0=\delta_0(n,r,V)$ depending only on $n$, $r$, and $V$ (but independent of $\varepsilon$).
In view of Definition \ref{dfn:str}, only such an upper bound matters when proving something for strainers.
The actual value of $\delta_0$ will be determined by the proof of Theorem \ref{thm:ind} (see Remark \ref{rem:d}).
On the other hand, the upper bound for $\varepsilon$ is $r/10$ as before.
\end{conv}

The following notion plays an essential role in the rescaling theorem.

\begin{dfn}[cf.\ {\cite[3.5]{Pet:int}}]\label{dfn:bad}
Let $N\in\bar{\mathcal M}(n,r,V)$, $M\subset N$, $p\in M$.
We say that $p$ is \textit{$\delta$-good} in $M$ if there exists $q\in M$ that is a $(1,\delta,\varepsilon)$-strainer on $\partial B(q,|pq|)$ for some $\varepsilon>0$ (depending on $p$).
Otherwise $p$ is \textit{$\delta$-bad} in $M$.
\end{dfn}

Later we will take $M$ as a fiber of a strainer map, but here $M$ is just a subset.
The point of the above definition is that $q\in M$ is a strainer at $p$, and being a strainer on the other part of $\partial B(q,|pq|)$ is not so essential.
However, it is suitable for the following rescaling theorem and also convenient for applying Proposition \ref{prop:ind} to an annulus around $\partial B(q,|pq|)$ in $M$ later, so we defined it this way.

Now we are ready to state the rescaling theorem.
Let $A(p;s,r)$ denote the open metric annulus $B(p,r)\setminus\bar B(p,s)$ (possibly $s=0$, i.e., $B(p,r)\setminus\{p\}$).

\begin{thm}[cf.\ {\cite[3.6]{Pet:int}}]\label{thm:res}
Let $N_j\in\mathcal M(n,r,V)$, $M_j\subset N_j$ closed, and $p_j\in M_j$.
Suppose $N_j$ converges to a GCBA space N, $M_j\to M\subset N$ closed, and $p_j\to p\in M$.
Assume $p$ is a $C\delta$-bad point in $M\cap\bar B(p,\rho)$ for sufficiently small $\rho>0$ (depending only on $p$), where $C=C(n,r,V)$ is a constant depending only on  $n$, $r$, and $V$.
Then there exist $\hat p_j\in M_j$ converging to $p$ and $\sigma_j\ge 0$ converging to $0$ such that the following hold:
\begin{enumerate}
\item $\hat p_j$ is a $(1,C\delta,\delta|\hat p_jx|)$-strainer at any $x\in A(\hat p_j;\sigma_j,\rho)$.
\item if $\sigma_j>0$ for any sufficiently large $j$, then there is a convergent subsequence of $(\sigma_j^{-1}N_j,\sigma_j^{-1}M_j,\hat p_j)$, whose limit $(\hat N,\hat M,\hat p)$ satisfies
\[\sup_{x\in\hat M}\pack\delta(\Sigma_x)\le\pack\delta(\Sigma_p)-1.\]
\end{enumerate}
\end{thm}

Note that if $M_j$ is a fiber of a $(k,\delta,\varepsilon)$-strainer map and $\rho<\varepsilon$, then Lemma \ref{lem:str3} allows us to extend the strainer map on the annulus around $\hat p_j$, and the remaining small ball where the strainer map cannot be extended, after rescaling, converges to a new limit space with smaller spaces of directions.

The rough idea of the proof is as follows.
A small neighborhood of $p$ looks like a metric cone, as in Lemma \ref{lem:str1}.
On the other hand, by the upper semicontinuity \cite[5.13]{LN:geo}, the space of directions at $q$ near $p$ is not bigger than that at $p$.
If it is not smaller, then the neighborhood of $q$ of the same radius must also look like a metric cone; otherwise, by Lemma \ref{lem:str2} we will get a contradiction.

The logical structure of the following proof is exactly the same as in the CBB case.
However, there is a crucial difference in Step 3, where we use the Riemannian structure of $N_j$.
See also Remark \ref{rem:res} and Example \ref{ex:res} below.

\begin{proof}
\step{Step 1}
First we define $\rho$ and $C$.
Choose $0<\rho<r/10$ so small that the rescaled ball $\rho^{-1}B(p,10\rho)$ is sufficiently close in the Gromov-Hausdorff distance to the $10$-ball around the vertex of the tangent cone $T_p$.
In particular, we may assume that the maximal number of points in $\partial B(p,\rho)$ making comparison angles at $p$ not less than $\delta$ is exactly $\pack\delta(\Sigma_p)$.
Indeed, the former is clearly not bigger than the latter for sufficiently small $\rho$, and the opposite inequality follows from geodesic completeness and angle comparison.
Furthermore, by Lemma \ref{lem:str1}, we may assume that $p$ is a $(1,100\delta,10\delta|px|)$-strainer at any $x\in B(p,\rho)\setminus\{p\}$.

Choose $C=C(n,r,V)$ large enough compared to the uniform doubling constant for the class $\mathcal M(n,r,V)$ (see Section \ref{sec:gcba}).
It will be determined in Step 3.

\step{Step 2}
Next we define $\sigma_j$ and $\hat p_j$ satisfying (1).
For $q\in\bar B(p_j,\rho)\cap M_j$, let
\[\sigma(q):=\min\{\sigma\ge0\mid\text{$q$ is a $(1,C\delta,\delta|qx|)$-strainer at $x\in A(q;\sigma,\rho)$}\}.\]
Note that $\sigma(p_j)\to 0$ as $j\to\infty$: otherwise, by the liftability of a strainer, Remark \ref{rem:str}(2), we get a contradiction to the above choice of $\rho$ (assume $C\gg100$).

Let $\sigma_j$ be the minimum of $\sigma(q)$ for $q\in\bar B(p_j,\rho)\cap M_j$, which certainly exists by the stability of strainers, Remark \ref{rem:str}(1).
We define $\hat p_j$ as one of the minimum points of $\sigma$.
Then clearly (1) holds.
Furthermore, $\hat p_j$ converges to $p$: otherwise, since $\sigma_j\to 0$ and strainers are stable, we get a contradiction to the assumption that $p$ is a $C\delta$-bad point in $\bar B(p,\rho)\cap M$.
This is the only place where we use the bad point assumption.

\step{Step 3}
Finally we prove (2).
Even after rescaling, the doubling constants of $N_j$ are still uniformly bounded above by the same constant.
Thus, passing to a subsequence, we may assume that $(\sigma_j^{-1}N_j,\hat p_j)$ converges to a locally compact geodesically complete CAT($0$) space $(\hat N,\hat p)$ in the pointed Gromov-Hausdorff topology.
Furthermore, we may assume that $\sigma_j^{-1}M_j$ converges to $\hat M\subset\hat N$.

Let $x\in\hat M$.
Take $x_j\in\sigma_j^{-1}M_j$ converging to $x$.
Note that $x_j$ converges to $p$ in the original convergence $N_j\to N$ since $\hat p_j\to p$ and $\sigma_j\to 0$.
Since $\sigma(x_j)\ge\sigma_j$ by definition, there exists $y_j\in B(x_j,\sigma(x_j))$ at which $x_j$ is not a $(1,C\delta,\delta|x_jy_j|)$-strainer and with $|x_jy_j|\ge\sigma_j/2$.
Note that $y_j$ converges to $p$ as well in the original convergence: indeed, since $x_j\to p$, the same argument as in the first paragraph of Step 2 shows $\sigma(x_j)\to 0$.

For $z\in B(x_j,\rho)\setminus\{x_j\}$, let $\bar z$ denote a point at distance $\rho$ from $x_j$ on a (unique) extension of the shortest path $x_jz$ beyond $z$.
By Lemma \ref{lem:str2}, there exists $z_j\in B(y_j,\delta|x_jy_j|)$ such that $\bar z_j$ is at distance $>C\rho\delta/2$ from $\bar y_j$ (note $|x_jy_j|\to0$).
Let $\gamma(t)$ be the shortest path between $y_j$ and $z_j$.
Then $\bar\gamma(t)$ is a continuous curve between $\bar y_j$ and $\bar z_j$ of length $>C\rho\delta/2$.
Therefore one can find more than $C/4$ points $\bar\gamma(t_i)$ with pairwise distance $>2\rho\delta$.
Here we used the assumption that $N_j$ is a Riemannian manifold for applying the geodesic flow to $\gamma$.

Set $A:=\pack\delta(\Sigma_x)$.
By geodesic completeness and angle monotonicity, one can find $\{w_\alpha\}_{\alpha=1}^A\subset\partial B(x,1/2)$ such that $\tilde\angle w_\alpha xw_{\alpha'}\ge\delta$ for any $\alpha\neq\alpha'$.
Take $w_\alpha^j\in\partial B(x_j,\sigma_j/2)\subset N_j$ converging to $w_\alpha$ in the new convergence $\sigma_j^{-1}N_j\to\hat N$.
Then we have $\tilde\angle w_\alpha^j x_jw_{\alpha'}^j\ge\delta_j$, where $\delta_j\nearrow\delta$ as $j\to\infty$.

Extend the shortest path $x_jw_\alpha^j$ beyond $w_\alpha^j$ and let $\underline w_\alpha^j $ and $\overline w_\alpha^j$ be points at distances $|x_jy_j|$ and $\rho$ from $x_j$ on this extension, respectively (note $|x_jy_j|\ge\sigma_j/2$).
Then by angle monotonicity, we have
\[\tilde\angle\overline w_\alpha^jx_j\overline w_{\alpha'}^j\ge\tilde\angle\underline w_\alpha^jx_j\underline w_{\alpha'}^j\ge\delta_j.\]
In particular, for sufficiently large $j$, the number of $\underline w_\alpha^j$ contained in the $4\delta|x_jy_j|$-neighborhood of $y_j$ is uniformly bounded above in terms of the doubling constant.
We denote this number by $A'=A'(n,r,V)$ and may assume that
\begin{alignat*}{2}
&\underline w_\alpha^j\in B(y_j,4\delta|x_jy_j|),\quad&1\le\alpha\le A'\\
&\underline w_\alpha^j\notin B(y_j,4\delta|x_jy_j|),\quad&A'<\alpha\le A.
\end{alignat*}
Since $\gamma(t)\in B(y_j,\delta|x_jy_j|)$, angle monotonicity shows that $\overline w_\alpha^j$ ($A'<\alpha\le A$) is at distance $>2\rho\delta$ from $\bar\gamma(t)$ for any $t$.
Recall that in Step 1, $C$ was chosen to be large enough compared to the doubling constant, in particular we may assume $C\gg A'$.
Therefore, combining the $2\rho\delta$-discrete points $\bar\gamma(t_i)$ more than $C/4$ with such $\overline w_\alpha^j$ ($A'<\alpha\le A$), we get more than $A$ points in $\partial B(x_j,\rho)$ making comparison angles at $x_j$ not less than $\delta_j$.
Taking $j\to\infty$ and recalling the definition of $\rho$ in Step 1, we obtain $\pack\delta(\Sigma_p)\ge A+1$, as required.
\end{proof}

\begin{rem}\label{rem:res}
Unlike the CBB case in \cite[\S3]{Pet:int}, the above rescaling theorem does not hold for general GCBA spaces.
See the following counterexample.
In fact, we used the assumption that $N_j$ are Riemannian manifolds in Step 3 when applying the geodesic flow to $\gamma$.
It might be interesting to find a proper subclass of GCBA spaces for which the above theorem (or this geodesic flow trick) holds.

The idea behind this trick is to compare the volume of a neighborhood of $y_j$ with that of $\bar\gamma$ (inspired by \cite[1.2]{Li:lvg}).
The former is uniformly bounded above by the relative comparison \`a la Bishop-Gromov, whereas the latter is bounded below in terms of the distance between the endpoints of $\bar\gamma$, since $\bar\gamma$ is continuous, by the absolute comparison of G\"unther (see Section \ref{sec:gcba}).
Furthermore, the distance between these endpoints is controlled by the divergence rate of the geodesic flow, i.e., the strainer condition \eqref{eq:str1}, as in Lemma \ref{lem:str2}.
Therefore, if $x_j$ is not a strainer at $y_j$, the volume of a neighborhood of $y_j$ will be less than that of $\bar\gamma$ by some fixed amount.
The Riemannian structure was used to ensure the continuity of $\bar\gamma$.
\end{rem}

\begin{ex}\label{ex:res}
Let $M_j$ be the gluing of two Euclidean spaces along an interval $I_j$ of length $1/j$ contained in them, which is a geodesically complete CAT($0$) space.
This sequence converges to the one-point gluing of the Euclidean spaces, denoted by $M$.
If $p\in M$ is the gluing point, which is bad in $M$, then $\hat p_j\in M_j$ of Theorem \ref{thm:res} should be the midpoint of $I_j$ with $\sigma_j=1/2j$.
The rescaled limit $\hat M$ will be the gluing of the Euclidean spaces along an interval $\hat I$ of length $2$, where $\hat p_j$ converges to the midpoint $\hat p$ of $\hat I$.
However, $\Sigma_p$ is the disjoint union of two unit spheres, whereas $\Sigma_{\hat p}$ is the gluing of the spheres at antipodal points.
Since they have the same volume, this generally contradicts to the conclusion (2) of Theorem \ref{thm:res}.
\end{ex}

\section{Proof}\label{sec:prf}

For a GCBA space $N\in\bar{\mathcal M}(n,r,V)$, we say that $M\subset N$ is a \textit{$(k,\delta,\varepsilon)$-strained surface} if it is a fiber of a $(k,\delta,\varepsilon)$-strainer map defined on its open $\varepsilon$-neighborhood in $N$ (in particular, $M$ is compact).
Recall that if $N$ is a Riemannian manifold, then $M$ is a smooth submanifold (see the end of Section \ref{sec:str}).

For a Riemannian manifold $M$, the sectional, Ricci, and scalar curvatures are denoted by $K$, $\Ric$, and $\Sc$, respectively, with subscript $M$ if necessary.
For $x\in M$ and $u\in\Sigma_x$, we will denote by
\begin{gather*}
K^\pm(x)=\max_{\sigma\subset T_x}\{\pm K(\sigma),0\},\quad\Sc^\pm(x)=\max\{\pm\Sc(x),0\},\\\Ric^\pm(u,u)=\max\{\pm\Ric(u,u),0\},
\end{gather*}
where $\sigma$ runs over all $2$-dimensional subspaces in $T_x$.

We prove the following theorem by reverse induction on $k$.
Recall that we are assuming $\delta<\delta_0=\delta_0(n,r,V)$ and $\varepsilon<r/10$ (see Convention \ref{conv:d}).

\begin{thm}[cf.\ {\cite[4.2]{Pet:int}}]\label{thm:ind}
Let $M$ be a $(k,\delta,\varepsilon)$-strained surface in a Riemannian manifold $N\in\mathcal M(n,r,V)$, where $k\le n-2$.
Then we have
\[\int_M\Sc^-_M<C\left(1+\int_MK^+_M\right),\]
where $C=C(n,r,V,\varepsilon)$ is a constant depending only on $n$, $r$, $V$, and $\varepsilon$.
\end{thm}

The $k=0$ case, i.e., $M=N$, together with the upper bounds on curvature and volume, implies Theorem \ref{thm:main}.
As mentioned in Convention \ref{conv:d}, the upper bound $\delta_0$ for $\delta$ is determined by the following proof.
More precisely, there will be other constants depending $\delta$, such as $C\delta$ or $\varkappa(\delta)$, and we choose $\delta_0$ so small that all these constants will be less than some dimensional constant (see Remark \ref{rem:d}).

\begin{rem}\label{rem:ind}
To be precise, we need a slightly generalized version of Theorem \ref{thm:ind} to logically complete the induction step.
There are two necessary generalizations.
\begin{enumerate}
\item Modifying the definition of a strained surface.
\item Localizing the statement of uniform boundedness.
\end{enumerate}
However, for simplicity, we defer them to Remarks \ref{rem:ind1} and \ref{rem:ind2}, respectively.
\end{rem}

\begin{rem}
The original CBB version \cite[4.2]{Pet:int} of Theorem \ref{thm:ind} was formulated for \textit{corner surfaces}, which is a subclass of modified strained surfaces.
More precisely, this requires an acute intersection angle condition for the level sets of modified strainer functions (\cite[3.1(iv)]{Pet:int}).
Our proof also shows that this assumption is not necessary for the statement (though the same technique is needed for the proof).
This is basically a matter of preference and is not an essential difference.
\end{rem}

The proof of Theorem \ref{thm:ind} has three steps.
In Step 1, we prove the base case of reverse induction.
In Step 2, we prove Proposition \ref{prop:ind} needed for the next step.
This proposition concerns a fixed manifold and has nothing to do with convergence.
Finally in Step 3, we show the induction step, carrying out another induction on the packing number, based on the rescaling theorem \ref{thm:res}.

\begin{proof}
\step{Step 1}
We first prove the base case $k=n-2$ by contradiction (cf.\ \cite[4.4]{Pet:int}). 
Suppose there exists a sequence of Riemannian manifolds $N_j\in\mathcal M(n,r,V)$ and $(n-2,\delta,\varepsilon)$-strained surfaces $M_j\subset N_j$ such that
\[\lim_{j\to\infty}\frac{\int_{M_j}\Sc^-_{M_j}}{1+\int_{M_j}K^+_{M_j}}=\infty.\]
Passing to a subsequence, we may assume that $N_j$ converges to a GCBA space $N$ and $M_j$ converges to an $(n-2,\delta,\varepsilon)$-strained surface $M\subset N$ (by the stability of strainers, Remark \ref{rem:str}(1), and the openness of strainer maps, \cite[8.2]{LN:geo}, if $\delta\ll1/n$).

Since $M_j$ are $2$-dimensional manifolds, the Gauss-Bonnet theorem gives
\[4\pi\chi(M_j)=\int_{M_j}\Sc_{M_j}=\int_{M_j}2K_{M_j}^+-\int_{M_j}\Sc_{M_j}^-\]
where $\chi$ denotes the Euler characteristic.
On the other hand, a stability result of Lytchak-Nagano \cite[13.1]{LN:geo}, based on Petersen's theorem \cite{Peter}, tells us that $M_j$ is homotopy equivalent to $M$ for sufficiently large $j$, if $\delta\ll1/n$ (actually it turns out that $M$ is a topological manifold, \cite[6.3]{LN:top}, and hence homeomorphic to $M_j$ by the $\alpha$-approximation theorem from geometric topology, \cite[4.7]{LN:top}).
Therefore we obtain a uniform bound on the value $(1+\int_{M_j}K^+_{M_j})^{-1}\cdot\int_{M_j}\Sc^-_{M_j}$, a contradiction.

\step{Step 2}
To prove the induction step in Step 3, we need the following proposition.
This gives a local estimate similar to the one of Theorem \ref{thm:ind} on a strained annulus in our strained surface.
In Step 3, by using the rescaling theorem \ref{thm:res}, we will cover our strained surface by a uniformly finite number of such strained annuli, whence the global estimate follows (actually by contradiction).

\begin{prop}[cf.\ {\cite[4.3 $\textrm B_k$]{Pet:int}}]\label{prop:ind}
Suppose Theorem \ref{thm:ind} holds for $k+1$, where $k\le n-3$.
Let $M$ be a $(k,\delta,\varepsilon)$-strained surface in $N\in\mathcal M(n,r,V)$.
Assume $p\in M$ is a $(1,\delta,\varepsilon'|px|)$-strainer at any $x\in A(p;\sigma,\rho)\cap M$, where $0\le4\sigma<\rho<\varepsilon$ and $\varepsilon'\le\delta$.
Then we have
\begin{equation}\label{eq:ind}
\int_{A(p;2\sigma,\rho/2)\cap M}\Sc_M^-<C\left(1+\int_{A(p;\sigma,\rho)\cap M}K_M^+\right),
\end{equation}
where $C=C(n,r,V,\varepsilon')$.
\end{prop}

We prove it by using the coarea, Gauss, and Bochner formulas as in \cite[4.5]{Pet:int}.
However, as explained in Section \ref{sec:ad}, our calculations are somewhat different from those in \cite{Pet:int}, due to the difference between the semiconcavity and (semi)convexity of distance functions in CBB and CBA manifolds (see Claims \ref{clm:conv} and \ref{clm:G}).

We first modify the distance function from $p$ slightly, as in the first formula in \cite[4.5]{Pet:int}, so that the resulting level set intersects $M$ at an acute angle (Claim \ref{clm:grad}, cf.\ \cite[3.1(iv)]{Pet:int}).
This modification makes the resulting function still convex with respect to the intrinsic metric of $M$, as we will see in Claim \ref{clm:conv}.
Its necessity is the same as in the CBB case.

Let $F=(f_1,\dots,f_k)$ be the strainer map defining $M$ and $G=(g_1,\dots,g_k)$ the opposite strainer map.
We define
\begin{equation}\label{eq:f}
f:=|p\cdot|+\varkappa(\delta)\sum_{i=1}^k\left(g_i(\cdot)-g_i(p)\right),
\end{equation}
where $\varkappa(\delta)$ is a positive function depending only on $n$ such that $\varkappa(\delta)\gg\delta$ and $\varkappa(\delta)\to0$ as $\delta\to0$ (to be determined in Claims \ref{clm:grad} and \ref{clm:conv}).
Note that $\varkappa(\delta)$ can be treated in the same manner as $\delta$, except that it is much greater than $\delta$.
The function $f$ behaves like the distance function from $p$ in the sense that
\begin{equation}\label{eq:fp}
|f(x)-|px||\le k\varkappa(\delta)|px|
\end{equation}
for any $x\in N$.
Note that $f$ is smooth, regular (its gradient has norm almost $1$, provided $\delta$ is small enough), and strictly convex on $B(p,\varepsilon)\setminus\{p\}$ (see the end of Section \ref{sec:str}).
Furthermore, the gradient of $f$ has the following properties, especially the obtuse angle condition with those of $f_i$.

\begin{clm}\label{clm:grad}
The gradient of $f$ in $N$ makes an obtuse angle almost $\pi/2$ with that of $f_i$ for any $i$.
More precisely, if we denote these gradients by $\nabla f$ and $\nabla f_i$, respectively, then
\[-2\varkappa(\delta)<\langle\nabla f,\nabla f_i\rangle<-\varkappa(\delta)/2\]
on $A(p;\sigma,\rho)\cap M$.
In particular, the gradient of $f$ in $M$ has norm almost $1$.
\end{clm}

\begin{proof}
By Lemma \ref{lem:str3}, $(F,|p\cdot|)$ is a $(k+1,10\delta,\varepsilon'|px|)$-strainer map at any $x\in A(p;\sigma,\rho)\cap M$.
In particular, as we have seen in the end of Section \ref{sec:str}, $\nabla f_i$ is almost orthogonal to $\nabla |p\cdot|$.
Similarly, $\nabla f_i$ is almost opposite to $\nabla g_i$ and almost orthogonal to $\nabla g_j$ for any $j\neq i$.
Note that all the errors are bounded above by $100\delta$, i.e.,
\[|\langle\nabla f_i,\nabla g_i\rangle+1|<100\delta,\quad|\langle\nabla f_i,\nabla g_j\rangle|,|\langle\nabla f_i,\nabla|p\cdot|\rangle|<100\delta.\]
These imply the desired inequality because $\varkappa(\delta)\gg\delta$ and $\delta\ll1/n$.
\end{proof}

We will now use the same notation as in Section \ref{sec:boc}, by identifying $M$ and $f$ with those in that section (thus $m+1=n-k$).

\begin{conv}
From now on, we basically use the notation in $M$ instead of $N$.
For example, $\nabla f$ denotes the gradient of $f$ in $M$.
When we use the notation in $N$, it will be indicated explicitly, like $\nabla ^Nf$.
In particular, as in Section \ref{sec:boc},
\begin{itemize}
\item $L=L_t$ denotes the level set of $f$ (in $M$) for $t\in[\sigma,\rho]$.
\item $u=\nabla f/|\nabla f|$ denotes the unit normal field to $L$ (in $M$).
\item $G$ and $H$ denote the Gauss-type and mean curvatures for $L$ (in $M$).
\end{itemize}
\end{conv}

By Lemma \ref{lem:str3} and \eqref{eq:f}, $L_t$ is a slight perturbation of a $(k+1,10\delta,\varepsilon't/10)$-strained surface.
We refer to it as a \textit{modified strained surface} (cf.\ corner surface in \cite[3.1]{Pet:int}).
In what follows, we will apply the inductive assumption of Theorem \ref{thm:ind} to the rescaled surface $t^{-1}L_t$ (see \eqref{eq:ass}).
As mentioned in Remark \ref{rem:ind}, technically this requires two generalizations of Theorem \ref{thm:ind}.
However, we defer the details to the end of Step 2 (Remarks \ref{rem:ind1} and \ref{rem:ind2}) and proceed with the proof assuming these generalizations.

The above obtuse angle condition implies

\begin{clm}[cf.\ {\cite[4.5(ii)c]{Pet:int}}]\label{clm:conv}
The principal curvatures of $L$ in $M$ are positive.
In particular, the Gauss-type and mean curvatures $G$ and $H$ are positive. 
\end{clm}

\begin{proof}
By definition (see Section \ref{sec:boc}), it suffices to show that $\langle\nabla_vu,v\rangle>0$ for any $v\in T_xL$.
A direct calculation shows
\begin{align*}
\langle\nabla_v u,v\rangle
&=\frac{\langle\nabla_v\nabla f,v\rangle}{|\nabla f|}
=\frac{\Hess f(v,v)}{|\nabla f|}\\
&=\frac{\Hess^Nf(v,v)+\langle\II_M^N(v,v),\nabla^Nf\rangle}{|\nabla f|}.
\end{align*}
Here $\Hess^Nf$ denotes the Hessian of $f$ in $N$ and $\II_M^N:T_xM\times T_xM\to(T_xM)^\bot$ denotes the second fundamental form of $M$ in $N$.
Since $f$ is strictly convex in $N$, $\Hess^Nf$ is positive definite.
Similarly, the strict convexity of $f_i$ implies
\[\langle\II_M^N(v,v),\nabla^N f_i\rangle=-\Hess^Nf_i(v,v)<0\]
Recall that $\nabla^Nf_i$ are $100\delta$-almost orthonormal, as seen in the end of Section \ref{sec:str}.
Since $\varkappa(\delta)\gg\delta$, it easily follows from the obtuse angle condition of Claim \ref{clm:grad} that $\langle\II_M^N(v,v),\nabla^N f\rangle$ is positive (see Lemma \ref{lem:lin}).
Therefore $\langle\nabla_v u,v\rangle$ is positive.
\end{proof}

The above positivity of the principal curvatures makes the following argument simpler in many places, but actually it is not so essential.
Only the uniform lower bound is important, as the uniform upper bound was important in the CBB case.

The fact that $L$ is a convex hypersurface in $M$ implies

\begin{clm}[cf.\ {\cite[4.5 Trivial inequality (ii)]{Pet:int}}]\label{clm:G}
\[K_L^+\le K_M^++G\]
\end{clm}

\begin{proof}
Let $e_1,\dots e_m$ be an arbitrary orthonormal basis of $T_xL$.
Let $G_{ij}$ denote the Gauss-type curvature of $L$ in $M$ for the tangent plane $\sigma_{ij}$ spanned by $e_i$ and $e_j$, i.e.,
\[G_{ij}=K_L(\sigma_{ij})-K_M(\sigma_{ij}).\]
It suffices to show $G_{12}\le G$.
Since $L$ is a convex hypersurface in $M$ by Claim \ref{clm:conv}, we have $G_{ij}\ge 0$ for any $(i,j)$.
Summing the above equalities for all $(i,j)$ and applying the inequalities except for $(i,j)=(1,2)$, we obtain
\[\Sc_L-\Sc_M+2\Ric(u,u)\ge G_{12},\]
where the left-hand side equals $G$ by the Gauss formula \eqref{eq:gau}.
\end{proof}

Finally we give an estimate on the integral of $H$.
The opposite estimate was also used in the CBB case at the end of \cite[4.5]{Pet:int}.

\begin{clm}\label{clm:H}
There exists $\rho'\in[3\rho/5,4\rho/5]$ such that
\[\int_{L_{\rho'}}H\le C\rho^{m-1},\]
where $C=C(n,r,V,\varepsilon')$ and $m=\dim L_{\rho'}$.
\end{clm}

\begin{proof}
Since $H=\di u$, the divergence theorem and the coarea formula imply
\[\frac d{dt}\vol L_t=\int_{L_t}\frac H{|\nabla f|}.\]
Since the norm of the gradient of $f$ is almost $1$ by Claim \ref{clm:grad}, the right-hand side is almost equal to $\int_{L_t}H$.
On the other hand, since $L_t$ is a modified $(k+1,10\delta,\varepsilon't/10)$-strained surface, it follows from Proposition \ref{prop:vol}, Remark \ref{rem:vol1}, and rescaling that
\[\vol L_t\le Ct^m,\]
where $C=C(n,r,V,\varepsilon')$.
The proof of this fact is deferred to the next section.
Together with the differential equation and the almost equality mentioned above, this shows the claim.
\end{proof}

\begin{proof}[Proof of Proposition \ref{prop:ind}]
Since $L=L_t$ is a modified $(k+1,10\delta,\varepsilon't/10)$-strained surface for any $t\in[\sigma,\rho]$, applying the inductive assumption of (the generalized version of) Theorem \ref{thm:ind} for $k+1$ to the rescaled one $t^{-1}L_t$, we obtain
\begin{equation}\label{eq:ass}
\int_L\Sc_L^-<C\left(t^{m-2}+\int_L K_L^+\right)<C\left(1+\int_L K_L^+\right),
\end{equation}
where $C=C(n,r,V,\varepsilon')$.
In the following calculations, $C$ will denote various such constants depending only on $n$, $r$, $V$, and $\varepsilon'$.

By the Gauss formula \eqref{eq:gau},
\[\Sc_M=\Sc_L+2\Ric(u,u)-G.\]
Since $G>0$ by Claim \ref{clm:conv}, we have
\begin{align*}
\Sc^-_M&\le\Sc^-_L+2\Ric^-(u,u)+G\\
&=\Sc^-_L+2\Ric^+(u,u)-2\Ric(u,u)+G\\
&\le\Sc^-_L+2mK^+_M-2\Ric(u,u)+G
\end{align*}
(in fact the positivity of $G$ is not necessary to obtain a similar estimate; compare with the Trivial inequality (i) for $\Sc^+_M$ in \cite[4.5]{Pet:int}).
Integrating both sides, we get
\begin{equation}\label{eq:sc}
\int_{f^{-1}[\sigma',\rho']}\Sc^-_M\le\int_{f^{-1}[\sigma',\rho']}\left(\Sc^-_L+2mK^+_M-2\Ric(u,u)+G\right),
\end{equation}
where $\sigma':=3\sigma/2$ and $\rho'$ is as in Claim \ref{clm:H}.
We may assume that $\sigma>0$ since the statement for $\sigma=0$ follows by taking the limit.

By \eqref{eq:fp}, the left-hand side of \eqref{eq:sc} is greater than that of \eqref{eq:ind}, provided $\delta$ is small enough.
We will show that each term of the right-hand side of \eqref{eq:sc} is less than that of \eqref{eq:ind}.
For the second term, this immediately follows from \eqref{eq:fp} as above.
For the third term, we have
\begin{align*}
-\int_{f^{-1}[\sigma',\rho']}\Ric(u,u)&=-\int_{f^{-1}[\sigma',\rho']}G-\int_{L_{\sigma'}}H+\int_{L_{\rho'}}H\tag{Bochner formula \eqref{eq:boc}}\\
&\le C\rho^{m-1}.\tag{Claims \ref{clm:conv}, \ref{clm:H}}
\end{align*}
For the first term, we have
\begin{align*}
\int_{f^{-1}[\sigma',\rho']}\Sc^-_L&\le C\int_{\sigma'}^{\rho'}dt\int_L\Sc^-_L\tag{Coarea formula, Claim \ref{clm:grad}}\\
&\le C\int_{\sigma'}^{\rho'}dt\left(1+\int_LK^+_L\right)\tag{Inductive assumption \eqref{eq:ass}}\\
&\le C\int_{\sigma'}^{\rho'}dt\left(1+\int_L\left(K^+_M+G\right)\right)\tag{Claim \ref{clm:G}}\\
&\le C\rho+C\int_{f^{-1}[\sigma',\rho']}\left(K^+_M+G\right).\tag{Coarea formula}
\end{align*}
Hence it remains to estimate the integral of $G$, the same as the last term of \eqref{eq:sc}.
\begin{align*}
\int_{f^{-1}[\sigma',\rho']}G&=\int_{f^{-1}[\sigma',\rho']}\Ric(u,u)-\int_{L_{\sigma'}}H+\int_{L_{\rho'}}H\tag{Bochner formula \eqref{eq:boc}}\\
&\le\int_{f^{-1}[\sigma',\rho']}mK^+_M+C\rho^{m-1}\tag{Claims \ref{clm:conv}, \ref{clm:H}}
\end{align*}
Combining all the inequalities above, we obtain \eqref{eq:ind}.
\end{proof}

Note that the last estimate on the integral of $G$ is one of the main differences from the CBB case discussed in Section \ref{sec:ad} (compare with the Trivial inequality (i) for $G$ in \cite[4.5]{Pet:int}).

We conclude Step 2 by discussing the two necessary generalizations of Theorem \ref{thm:ind} mentioned in Remark \ref{rem:ind}, so that the inductive assumption can certainly be applied to $t^{-1}L_t$ to obtain \eqref{eq:ass}.
The first one is the modification of a strained surface and the second one is the localization of the statement.
We will only outline them and leave the details to the reader.

\begin{rem}\label{rem:ind1}
Firstly, we have to formulate Theorem \ref{thm:ind} for a modified strained surface defined by modified strainer functions of the form \eqref{eq:f}.
This will be done as follows, for example.
Let $F=(f_1,\dots,f_k)$ be a $(k,\delta,\varepsilon)$-strainer map at $p\in N$ with an opposite strainer map $G=(g_1,\dots,g_k)$.
We define a \textit{modified strainer map} $\tilde F=(\tilde f_1,\dots,\tilde f_k)$ by
\[\tilde f_i:=f_i(\cdot)+\varkappa_{ij}(\delta)\sum_{j=1}^k(g_j(\cdot)-g_j(p)),\]
where $\varkappa_{ij}(\delta)$ are nonnegative functions tending to $0$ as $\delta\to0$ (possibly $\varkappa_{ij}\equiv 0$), which are uniformly bounded above by some fixed function $\bar\varkappa(\delta)$ with the same property.
We formulate Theorem \ref{thm:ind} for such a modified strained map from the beginning of the reverse induction.
This modification leaves the level set of a strainer function nearly unchanged (in particular their almost orthogonality), provided $\delta$ is small enough, and preserves the strict convexity of a strainer function.
Thus all the properties of a strained surface we used remain true, such as the openness and homotopical stability \cite[8.2, 13.1]{LN:geo} and Lemma \ref{lem:str3}.
Then one can define a strictly convex function $f$ as in \eqref{eq:f}, to which the inductive assumption can be applied, and for which Claims \ref{clm:grad} and \ref{clm:conv} hold true, provided $\varkappa(\delta)\gg\bar\varkappa(\delta)$.
Note that the upper bound $\bar\varkappa(\delta)$ will be determined inductively, especially based on Lemma \ref{lem:lin} used in the proof of Claim \ref{clm:grad}.
\end{rem}

\begin{rem}\label{rem:ind2}
Secondly, we have to localize Theorem \ref{thm:ind} so that it can be applied to the rescaled surface $t^{-1}L_t$, because our manifold has no global volume bound after such rescaling.
However, the rescaled manifold still has a local volume bound by the relative volume comparison \`a la Bishop-Gromov (see Section \ref{sec:gcba}).
Therefore the necessary modification will be as follows.
Consider the family of $n$-dimensional Riemannian manifolds with base points (not necessarily compact), such that on the $R$-ball around the base point, it holds that sectional curvature $\le1$, injectivity radius $\ge r>0$, and volume $\le V$.
Then one can bound the curvature integral of a (modified) strained surface contained in the $R/2$-ball around the base point in terms of a constant depending also on $R$.
\end{rem}

\step{Step 3}
We complete the induction step of the proof of Theorem \ref{thm:ind}.
The proof is by contradiction.
Suppose Theorem \ref{thm:ind} holds for $k+1$ but not for $k$, where $k\le n-3$.
More precisely, we fix sufficiently small $\delta>0$ such that Proposition \ref{prop:ind} holds true with $\delta$ replaced by $C\delta$, where $C=C(n,r,V)$ is a constant from Theorem \ref{thm:res}, and assume that Theorem \ref{thm:ind} for $k$ does not hold for this fixed $\delta$.

By precompactness, we may assume that there are $N_j\in\mathcal M(n,r,V)$ converging to a GCBA space $N$ and $(k,\delta,\varepsilon)$-strained surfaces $M_j\subset N_j$ converging to $M\subset N$ such that 
\[\lim_{j\to\infty}\frac{\int_{M_j}\Sc_{M_j}^-}{1+\int_{M_j}K_{M_j}^+}=\infty.\]
By the compactness of $M$, the problem reduces to the following local claim.

\begin{clm}[cf.\ {\cite[4.6]{Pet:int}}]\label{clm:lim}
For any $p\in M$, there exists $\rho>0$ such that
\[\liminf_{j\to\infty}\frac{\int_{B(p_j,\rho/2)\cap M_j}\Sc_{M_j}^-}{1+\int_{M_j}K_{M_j}^+}<\infty\]
for some (and thus any) $p_j\in M_j$ converging to $p$.
\end{clm}

\begin{proof}
We prove it by induction on the value
\[\sup_{p\in M}\pack\delta(\Sigma_p)\]
for the above fixed $\delta$.
Note that this value is uniformly bounded from below and above, as observed at the beginning of Section \ref{sec:res}.
Let $0<\rho<\varepsilon$ be small enough as in the proof of Theorem \ref{thm:res}, so that in particular Lemma \ref{lem:str1} holds.
The proof is divided into two cases:
\begin{enumerate}
\item $p$ is $C\delta$-good in $M\cap\bar B(p,\rho)$; or
\item $p$ is $C\delta$-bad in $M\cap\bar B(p,\rho)$,
\end{enumerate}
where Theorem \ref{thm:res} will be applied to the latter (see Definition \ref{dfn:bad} for good and bad points).

First we consider the good case (1).
By definition, there exists $q\in M\cap\bar B(p,\rho)$ that is a $(1,C\delta,\varepsilon')$-strainer on $\partial B(q,|pq|)$ for some $\varepsilon'>0$.
Since Lemma \ref{lem:str2} also holds, $B(p,\rho/2)$ is covered by two metric annuli $A(p;2\sigma,\rho/2)$ and $A(q;2\sigma',\rho'/2)$ such that $p$ and $q$ are strainers on $A(p;\sigma,\rho)$ and $A(q;\sigma',\rho')$, respectively, as in Proposition \ref{prop:ind} with $\delta$ replaced by $C\delta$.
Take arbitrary sequences $p_j,q_j\in M_j$ converging to $p,q$.
Then by Remark \ref{rem:str}(2) and Proposition \ref{prop:ind}, we have
\begin{align*}
\int_{B(p_j,\rho/2)\cap M_j}\Sc_{M_j}^-&<\int_{A(p_j;2\sigma,\rho/2)\cap M_j}\Sc_{M_j}^-+\int_{A(q_j;2\sigma',\rho'/2)\cap M_j}\Sc_{M_j}^-\\
&<C'\left(1+\int_{M_j}K_{M_j}^+\right)
\end{align*}
for any sufficiently large $j$, where $C'$ is another constant independent of $j$.

Next we consider the bad case (2).
Let $\hat p_j$, $\sigma_j$ be as in Theorem \ref{thm:res}.
As above, by Theorem \ref{thm:res}(1) and Proposition \ref{prop:ind}, we have
\[\int_{A(\hat p_j,2\sigma_j,\rho/2)\cap M_j}\Sc_{M_j}^-<C'\left(1+\int_{M_j}K_{M_j}^+\right).\]
If $\sigma_j\equiv0$ for some subsequence, then we are done.
Thus we may assume $\sigma_j>0$ for any large $j$, as in Theorem \ref{thm:res}(2).
Passing to a subsequence, we obtain a new limit $(\hat N,\hat M,\hat p)$ of $(\sigma_j^{-1}N_j,\sigma_j^{-1}M_j,\hat p_j)$ with
\[\sup_{x\in\hat M}\pack\delta(\Sigma_x)\le\pack\delta(\Sigma_p)-1.\]
In particular, since the left-hand side is uniformly bounded below by $\pack\delta(\mathbb S^{n-1})$, the base case of the induction must be included in the previous two cases (i.e., $p$ is good or $\sigma_j\equiv 0$; notice that the condition of the base case implies that $M$ is almost a Riemannian manifold).
By the inductive assumption and the compactness of $\bar B(\hat p,2)\cap\hat M$, we get
\[\liminf_{j\to\infty}\frac{\int_{B(\hat p_j,2\sigma_j)\cap M_j}\Sc_{M_j}^-}{\sigma_j^{m-1}+\int_{M_j}K_{M_j}^+}=\liminf_{j\to\infty}\frac{\int_{\sigma_j^{-1}B(\hat p_j,2\sigma_j)\cap M_j}\Sc_{\sigma_j^{-1}M_j}^-}{1+\int_{\sigma_j^{-1}M_j}K_{\sigma_j^{-1}M_j}^+}<\infty,\]
where $m+1=\dim M_j$ (strictly speaking, as in Remark \ref{rem:ind2}, we need the localized version of Claim \ref{clm:lim} when using it as the inductive assumption).
Together with the previous estimate on the annulus, this completes the proof.
\end{proof}

The proof of Theorem \ref{thm:ind} is now complete.
\end{proof}

\begin{rem}\label{rem:d}
Let us finally summarize how to choose the constant $\delta_0=\delta_0(n,r,V)$ of Convention \ref{conv:d}.
Following the above proof, we will define a decreasing reverse sequence of constants $\delta_k$ ($0\le k\le n-2$) inductively such that Theorem \ref{thm:ind} for $k$ holds provided $\delta<\delta_k$.
Step 1 shows that $\delta_{n-2}$ can be chosen to be a constant depending only on $n$, for which the results of Lytchak-Nagano \cite[8.2, 13.1]{LN:geo} hold.
Suppose $\delta_{k+1}$ was chosen.
The argument in Step 2 shows that Proposition \ref{prop:ind} holds provided $\delta\ll\delta_{k+1}$ (it depends especially on Lemma \ref{lem:lin} used in Claim \ref{clm:conv}).
In Step 3, we fixed small $\delta$ such that Proposition \ref{prop:ind} with $\delta$ replaced by $C\delta$ holds, where $C=C(n,r,V)$ is the constant of Theorem \ref{rem:res}.
Combining these two steps shows that $\delta_k$ can be chosen to be a constant much less than $C^{-1}\delta_{k+1}$.
After such $(n-2)$ steps, one can determine the actual value of $\delta_0$ explicitly.
In particular, the constant $\delta_0$ comes essentially from the doubling constant for our class $\mathcal M(n,r,V)$, as so does the constant $C$ of Theorem \ref{thm:res}.
\end{rem}

\section{Appendix}\label{sec:app}

We show two auxiliary results used in the previous section.
The first one is an elementary lemma from linear algebra, which was used in the proof of Claim \ref{clm:conv} (and also in Remarks \ref{rem:ind1} and \ref{rem:d}).

\begin{conv}
Here we temporarily assume that $\delta$ is a positive number less than some small constant depending only on $n$.
We also denote by $\varkappa(\delta)$ various positive functions depending only on $n$ such that $\varkappa(\delta)\gg\delta$ and $\varkappa(\delta)\to0$ as $\delta\to0$, with additional symbols to distinguish different ones.
\end{conv}

The lemma asserts the existence of such $\varkappa_1\gg\varkappa_2$ satisfying the statement.

\begin{lem}\label{lem:lin}
Let $E$ be an $n$-dimensional vector space with inner product $\langle\ ,\ \rangle$.
Suppose $u_1,\dots,u_k\in E$ are $\delta$-almost orthonormal in the sense that
\[|\langle u_i,u_i\rangle-1|<\delta,\quad|\langle u_i,u_j\rangle|<\delta\]
for any $1\le i\neq j\le k$.
Let $x,y\in E$ be unit vectors such that
\begin{equation}\label{eq:lin}
\langle x,u_i\rangle<0,\quad\langle y,u_i\rangle<-\varkappa_1(\delta)
\end{equation}
for any $1\le i\le k$.
Then $\langle x,y\rangle>\varkappa_2(\delta)$.
\end{lem}

\begin{proof}
The proof is by induction on $k$.
The base case $k=1$ is trivial.
Suppose $k\ge 2$ and there exist such functions $\bar\varkappa_1\gg\bar\varkappa_2$ satisfying the statement for $k-1$.
We may assume $\varkappa_1\gg\bar\varkappa_1$ and $\varkappa_2\ll\bar\varkappa_2$.
Let $H$ denote the hyperplane spanned by $u_1,\dots,u_{k-1}$ and $H^\perp$ the orthogonal complement in $E$.
We will denote the corresponding orthogonal decomposition as
\[x=x_H+x^\perp,\quad y=y_H+y^\perp,\quad u_k=u_k^H+u_k^\perp.\]
Then we have
\[\langle x,y\rangle=\langle x_H,y_H\rangle+\langle x^\perp,y^\perp\rangle.\]
By the inductive assumption,
\[\langle x_H,y_H\rangle\ge\bar\varkappa_2(\delta)|x_H||y_H|\]
(possibly the norms are $0$).
The $\delta$-almost orthonormality implies $|u_k^H|<C\delta$ and $|u_k^\perp|>1-C\delta$, where $C$ denotes a positive constant depending only on $k$.
This, together with the assumption \eqref{eq:lin}, easily yields
\[\langle x^\perp,u_k^\perp\rangle<C\delta,\quad|y^\perp|>\varkappa_1(\delta)/2,\quad|y_H|>\varkappa_1(\delta)/2\]
since $\varkappa_1(\delta)\gg C\delta$.
Note that $y^\perp$ is opposite to $u_k^\perp$, but this is not necessarily the case for $x^\perp$.
Hence the proof is divided into three cases:
\begin{enumerate}
\item $\langle x^\perp,y^\perp\rangle\ge0$ and $|x^\perp|\ge1/2$;
\item $\langle x^\perp,y^\perp\rangle\ge0$ and $|x^\perp|<1/2$;
\item $\langle x^\perp,y^\perp\rangle<0$.
\end{enumerate}
In the case (1), we have $\langle x^\perp,y^\perp\rangle>\varkappa_1(\delta)/4$.
In the case (2), we have $|x_H|\ge1/2$ and thus $\langle x_H,y_H\rangle>\bar\varkappa_2(\delta)\varkappa_1(\delta)/4$ by the inductive assumption.
In the case (3), since $x^\perp$ has the same direction as $u_k^\perp$, we have $|x^\perp|<2C\delta$ and $|x_H|>1-2C\delta$.
Therefore
\[\langle x,y\rangle=\langle x_H,y_H\rangle+\langle x^\perp,y^\perp\rangle>\bar\varkappa_2(\delta)\varkappa_1(\delta)/4-2C\delta.\]
Since $\varkappa_1(\delta)\gg\bar\varkappa_2(\delta)\gg\varkappa_2(\delta)\gg C\delta$, the claim holds in every case.
\end{proof}

\begin{rem}
If $\langle u_i,u_j\rangle<0$ instead of almost orthonormality, then it is easier to see that $\langle x,u_i\rangle<0$, $\langle y,u_i\rangle<0$ imply $\langle x,y\rangle>0$ (because $\langle x^\perp,y^\perp\rangle>0$).
This is enough to prove Claim \ref{clm:conv} if we work in the class of corner surfaces as in \cite[3.1]{Pet:int}, rather than the modified strained surfaces we used.
\end{rem}

The second one is a uniform upper bound for the volume of strained surfaces, which was used in the proof of Claim \ref{clm:H}.
Here we follow Convention \ref{conv:d} as before, i.e., $\delta<\delta_0(n,r,V)$ and $\varepsilon<r/10$ (but see also Remark \ref{rem:vol2} below).

\begin{prop}\label{prop:vol}
Let $M$ be a $(k,\delta,\varepsilon)$-strained surface in $N\in\mathcal M(n,r,V)$, where $k\le n$.
Then we have
\[\vol M\le C(n,r,V,\varepsilon).\]
\end{prop}

There are two proofs of the above proposition.
The first one is the analogy of the proof of Theorem \ref{thm:ind}, based on the rescaling theorem \ref{thm:res}.
The second one is much simpler and more general, which uses finiteness results of Lytchak-Nagano \cite[10.1, 10.5]{LN:geo}.
We will outline both the proofs below, but leave the details to the reader.
The author is grateful to Alexander Lytchak for providing the second proof.

\begin{rem}\label{rem:vol1}
The two generalizations mentioned in Remark \ref{rem:ind} are applicable to the above proposition as well.
Indeed, the first proof, the analogy of Theorem \ref{thm:ind}, requires the localization as in Remark \ref{rem:ind2} to complete the induction step.
However, unlike the previous case, the modification as in Remark \ref{rem:ind1} is unnecessary for the proof of Proposition \ref{prop:vol} because the latter does not need the obtuse angle condition (though the modification is possible).
\end{rem}

\begin{rem}\label{rem:vol2}
The constant $C$ actually does not depend on the straining radius $\varepsilon$ and also the upper bound $\delta_0$ for $\delta$ depends only on $n$.
These follow from the second simplified proof due to Lytchak.
Moreover, this proof extends directly to strained surfaces in general GCBA spaces, where the strainer is in the original weaker sense of Lytchak-Nagano \cite[7.2]{LN:geo}.
\end{rem}

\begin{proof}[Proof 1]
The proof is by reverse induction on $k$, which has exactly the same structure as the proof of Theorem \ref{thm:ind} in the previous section.
The base case $k=n$ immediately follows from the fact that any $(n,\delta,\varepsilon)$-strainer map is injective on an $\varepsilon$-ball, and hence its fiber is $\varepsilon$-discrete (actually it is a bi-Lipschitz open embedding; cf.\ \cite[11.2]{LN:geo}, Lemma \ref{lem:str3}).

To prove the induction step, suppose there are $N_j\in\mathcal M(n,r,V)$ and $(k,\delta,\varepsilon)$-strained surfaces $M_j\subset N_j$ such that $\vol M_j\to\infty$ as $j\to\infty$.
We may assume $N_j$ converges to a GCBA space $N$ and $M_j$ converges to $M\subset N$.

The proof proceeds along the same line as Step 3 of the proof of Theorem \ref{thm:ind}.
We use induction on the packing number of the space of directions of the limit space.
The difference is that we consider much simpler value $\vol=\int$ instead of $(1+\int K^+)^{-1}\cdot\int\Sc^-$.
The inductive assumption of the induction on $k$ will be used as follows.
Under the same assumptions as in Proposition \ref{prop:ind}, by using the coarea formula and the almost orthogonality condition of Claim \ref{clm:grad} (without the obtuse angle condition), we have
\[\vol A(p;\sigma,\rho)\le C\int_\sigma^\rho\vol L_tdt\le C\rho,\]
where the second inequality follows from the inductive assumption (and rescaling).
Using this inequality instead of \eqref{eq:ind} and arguing as in Step 3, one can complete the proof.
\end{proof}

\begin{proof}[Proof 2]
The proof is by reverse induction on $k$ as above.
First recall that our manifold $N$ can be covered by a uniformly finite number of tiny balls of radius $r/10$ (see Section \ref{sec:gcba}).
Thus it suffices to show the statement on each tiny ball.

Here we use the original infinitesimal definition of a strainer by Lytchak-Nagano \cite[7.2]{LN:geo} to complete the induction step.
We fix $\delta$, which will be determined later.
The base case $k=n$ follows directly from their finiteness result \cite[10.5]{LN:geo}.
Note that the \textit{capacity} in their paper (see \cite[\S5.2]{LN:geo}) is nothing but the doubling constant in our paper (see Section \ref{sec:gcba}).
Therefore it depends only on $n$, $r$, and $V$, and so does the upper bound for the volume of $M$ (thus independent of $\varepsilon$).

To prove the induction step, we use \cite[10.5]{LN:geo} again, which says that the strainer map $F$ defining $M$ can be extended except at (uniformly) finitely many points (as Lytchak-Nagano's strainer map).
Moreover, by \cite[10.1]{LN:geo}, the number of strainer functions $f_i$ needed to extend $F$ is uniformly bounded above in terms of the capacity.
Therefore $M$ is covered by a uniformly finite number of open subsets $U_i$ in $M$ where $(F,f_i)$ is a strainer map.
The claim now follows by applying the coarea formula to $f_i$ on $U_i$, together with the inductive assumption for the fibers of $(F,f_i)$.
In each induction step the error $\delta$ gets worse as in \cite[10.5]{LN:geo}, but eventually its upper bound $\delta_0$ can be chosen to be a constant depending only on $n$.
\end{proof}

\begin{rem}\label{rem:vol3}
The main difference between the two proofs is explained as follows.
Both are by reverse induction on $k$, so the basic idea is to extend a strainer map $F$ defining a strained surface $M$ by adding another strainer function $f$, as in Lemma \ref{lem:str3} or \cite[9.4]{LN:geo}, and to apply the coarea formula to $f|_M$ together with the inductive assumption for the fibers of $f|_M$.
The uniform boundedness of the volume follows if one could cover $M$ by a uniformly finite number of subsets $U_i$ where $F$ can be extended.
In the first proof, such subsets $U_i$ are annuli, as in Proposition \ref{prop:ind}, and the uniform finiteness of the covering comes from the rescaling theorem \ref{thm:res} and the precompactness.
However, in the second proof, $U_i$ are not annuli, but have more complicated shapes.
The uniform finiteness of the covering comes essentially from that of strainer functions involved in the extension of $F$ as in \cite[10.1]{LN:geo}.
The latter is only possible because of the infinitesimal definition of a strainer by Lytchak-Nagano and the GCBA condition, and has no (direct) counterpart in CBB geometry (however, compare with \cite[1.1]{MY:ob} for the noncollapsing case).
On the other hand, the advantage of the first proof is that it is applicable to (collapsing) Alexandrov spaces (cf.\ \cite[6.16]{F:fib}).
The author is not sure whether there is a simplified proof of Theorem \ref{thm:main} from this point of view.
\end{rem}

\end{document}